\documentclass[]{amsart}

\usepackage{graphicx} 
\usepackage{amsmath}
\usepackage{amsthm}
\usepackage{amsfonts}
\usepackage{amssymb}
\usepackage{enumitem}
\usepackage{comment}
\usepackage{mathtools}

\usepackage[style=numeric,sorting=nyt,backend=bibtex8, maxbibnames=99]{biblatex}
\bibliography{biblio.bib}

\usepackage{hyperref}
\usepackage{scalerel}
\newcommand\snwarrow{\mathrel{\scaleobj{0.7}{\nwarrow}}}
\newcommand\snearrow{\mathrel{\scaleobj{0.7}{\nearrow}}}
\newcommand\sswarrow{\mathrel{\scaleobj{0.7}{\swarrow}}}
\newcommand\ssearrow{\mathrel{\scaleobj{0.7}{\searrow}}}

\usepackage{xcolor}
\usepackage[author=L, status=final]{fixme}
\fxsetup{targetlayout=color}
\definecolor{fxtarget}{HTML}{007ACC}
\fxsetup{nomargin,nomarginclue,noinline,nofootnote,noindex}

\newtheorem{theorem}{Theorem}[]
\newtheorem{proposition}[theorem]{Proposition}
\newtheorem{corollary}[theorem]{Corollary}
\newtheorem{lemma}[theorem]{Lemma}
\newtheorem{claim}{Claim}[theorem]
\theoremstyle{definition}

\newtheorem*{theorem*}{Theorem}
\newtheorem*{corollary*}{Corollary}

\newtheorem{definition}[theorem]{Definition}

\theoremstyle{remark}

\begin{document}

\title{Open Colorings and Baumgartner's Axiom}

\author{Lorenzo Notaro}
\address{University of Vienna, Institute of Mathematics, Kurt G\"{o}del Research Center, Kolingasse 14-16, 1090 Vienna, Austria}
\curraddr{}
\email{lorenzo.notaro@univie.ac.at}
\date{}
\thanks{This research was funded in whole or in part by the Austrian Science Fund (FWF) \href{https://www.fwf.ac.at/en/research-radar/10.55776/ESP1829225}{10.55776/ESP1829225}. For open access purposes, the author has applied a CC BY public copyright license to any author accepted manuscript version arising from this submission.}

\subjclass[2020]{Primary 03E35, Secondary 03E05, 03E50}
\keywords{Baumgartner's Axiom, Entangled sets, Increasing sets, Open Coloring Axiom, Martin's Axiom}

\begin{abstract}
We construct a model of $\mathsf{MA_{\aleph_1}}+\mathsf{OCA}_T$ where Baumgartner's Axiom fails, settling a question of Farah \cite[Question~(a)]{MR1440716}. Moreover, in the same model there is an $\aleph_1$-dense set of reals which is neither reversible nor increasing, answering a question of Marun, Shelah, and Switzer \cite[Question~4.6]{marun2025baumgartnersaxiomsmallposets}.
\end{abstract}

\maketitle

\section{Introduction}

Todor\v{c}evi\'{c}'s Open Coloring Axiom, here denoted by $\mathsf{OCA}_{T}$, is the following statement:\vspace{0.5em}

\begin{center}
  \begin{minipage}{\dimexpr\linewidth-5em}
    For every separable metrizable space $S$ and every open graph $G \subseteq [S]^2$,  one of the following holds:
  \begin{enumerate}
  \item $\chi(G) \le \aleph_0$, or
  \item there exists an uncountable $G$-clique.
  \end{enumerate}\vspace{0.5em}
  \end{minipage}
\end{center}

This Ramsey-type dichotomy has a wide range of applications, many of which concern rigidity phenomena for quotient structures---for instance, quotient Boolean algebras, corona algebras, and the Calkin algebra \cite{MR4288216, MR2776359, MR4642565, MR4290495, MR1999941, MR1202874}. A number of these applications are obtained from the combination $\mathsf{MA}_{\aleph_1} + \mathsf{OCA}_{T}$.

The principle $\mathsf{OCA}_{T}$ is independent of $\mathsf{ZFC}$: it is false under $\mathsf{CH}$, it is consistent relative to $\mathsf{ZF}$, and it is implied by $\mathsf{PFA}$ \cite{MR980949, MR1233818}. It was inspired by similar principles introduced and studied by Abraham, Rubin, and Shelah \cite{MR801036}. One of these principles bears the same name; we denote it by  $\mathsf{OCA}_{ARS}$ (see Section~\ref{sec:ARS}).

In \cite{MR1440716}, Farah investigated why $\mathsf{MA}_{\aleph_1}$ is sometimes needed to complete arguments carried out under $\mathsf{OCA}_{T}$. In particular, he studied the relationship between $\mathsf{OCA}_{T}$ and Baumgartner's Axiom ($\mathsf{BA}$), which asserts that every two $\aleph_1$-dense sets of reals are order-isomorphic. Baumgartner's Axiom is a natural uncountable analogue of Cantor's isomorphism theorem, which says that any two countable dense linear orders without endpoints are isomorphic. Baumgartner proved that $\mathsf{BA}$ is consistent relative to $\mathsf{ZF}$ in his seminal paper \cite{MR317934} and that it follows from $\mathsf{PFA}$ \cite{MR776640}. 

Farah proved in \cite{MR1440716} that $\mathsf{BA}$ does not follow from $\mathsf{OCA}_{T}$ (see also \cite{MR1978945}), and asked whether $\mathsf{BA}$ follows from $\mathsf{MA_{\aleph_1}} + \mathsf{OCA}_T$ \cite[Question (a)]{MR1440716}. Here we answer his question in the negative:

\begin{theorem}\label{thm:mainintro}
It is consistent relative to $\mathsf{ZF}$ that $\mathsf{MA_{\aleph_1}}+\mathsf{OCA}_T$ holds, but both $\mathsf{BA}$ and $\mathsf{OCA}_{ARS}$ fail.
\end{theorem}

It is worth noting that $\mathsf{OCA}_{ARS}$ is already known not to follow from $\mathsf{OCA}_{T}$ alone \cite{MR4228344}.

Our result also fits into a line of work showing that certain fragments of $\mathsf{PFA}$ do not suffice to prove $\mathsf{BA}$ (e.g., \cite{MR801036, avraham1981martin, MR3385104, MR4819970, MR4128471}). We emphasize that, in these models, the failure of $\mathsf{BA}$ is witnessed by the existence of either a $2$-entangled set of reals or an increasing set (see Section~\ref{sec:entangled}); consequently, these are all models in which $\mathsf{OCA}_{T}$ also fails (see again Section~\ref{sec:entangled}).

The failure of $\mathsf{BA}$ in the model we construct to prove Theorem~\ref{thm:mainintro} is witnessed by an $\aleph_1$-dense set of reals which is not reversible (i.e., not isomorphic to its reverse, see Section~\ref{sec:notation}). Since $\mathsf{OCA}_T$ prevents the existence of increasing sets, our model also yields a negative answer to a question of Marun, Shelah, and Switzer \cite[Question 4.6]{marun2025baumgartnersaxiomsmallposets}, who asked whether $\mathsf{MA}_{\aleph_1}$ implies that every $\aleph_1$-dense set of reals is either reversible or increasing\footnote{In \cite[Question~4.6]{marun2025baumgartnersaxiomsmallposets} ``increasing" is replaced by ``good" \cite[Definition 1.5]{marun2025baumgartnersaxiomsmallposets}. For $\aleph_1$-dense sets of reals, these two notions are equivalent: an $\aleph_1$-dense set is good (in the sense of \cite{marun2025baumgartnersaxiomsmallposets}) if and only if it is increasing (in the sense of \cite{MR801036}).}:

\begin{corollary}\label{cor:intro}
It is consistent relative to $\mathsf{ZF}$ that $\mathsf{MA_{\aleph_1}}$ holds and there exists an $\aleph_1$-dense set of reals which is neither increasing nor reversible.
\end{corollary}

In Section~\ref{sec:prel}, we fix notation and recall the definitions of $\mathsf{OCA}_{ARS}$, $2$-entangled sets, and increasing sets. Then, in Section~\ref{sec:main}, we prove our result. We first define \emph{nonstationarily $2$-entangled} sequences of reals, which satisfy a property weaker than being $2$-entangled, but whose existence suffices to imply the failure of both $\mathsf{BA}$ and $\mathsf{OCA}_{ARS}$. In particular, the range of these sequences is a non-reversible $\aleph_1$-dense set of reals. We then force $\mathsf{MA_{\aleph_1}}+\mathsf{OCA}_T$ via a finite support iteration arranged so as to preserve the existence of a nonstationarily $2$-entangled sequence using the explicit contradiction method introduced by Abraham and Shelah in \cite{avraham1981martin}.

\section{Preliminaries}\label{sec:prel}

\subsection{Notation and Terminology}\label{sec:notation} Our set-theoretic notation is standard; see, e.g., \cite{MR1940513}. 

Given a set $X$ and an (undirected) graph $G\subseteq [X]^2$, we let $\chi(G)$ be the chromatic number of $G$. Moreover, if $X$ is a Hausdorff topological space, we say that $G \subseteq [X]^2$ is open if the set \[\{(a,b) \in X^2 \mid \{a,b\} \in G\}\] is open with respect to the product topology on $X^2$.

We write $f: X\rightharpoonup Y$ for a partial function $f$ whose domain is contained in $X$ and whose range is contained in $Y$. A real function $f:\mathbb{R} \rightharpoonup \mathbb{R}$ is said to be \emph{increasing} (resp. \emph{decreasing}) if it is order-preserving (resp. order-reversing)---i.e., $x \le y$ implies $f(x) \le f(y)$ (resp. $f(x) \ge f(y)$) for all $x,y \in \mathrm{dom}(f)$. A function is said to be \emph{monotone} if it is either increasing or decreasing.  A relation $R\subseteq \mathbb{R}^2$ is said to be \emph{non-decreasing} (resp. \emph{non-increasing}) if for all $(x,y), (x',y') \in R$, $x < x'$ implies $y \le y'$ (resp. $y \ge y'$).

Given a binary relation $R \subseteq X\times Y$ and some $x \in X$, we denote by $R_x$ the set $\{y \in Y \mid x \mathrel{R} y\}$. Similarly, given $y \in Y$,  we denote by $R^y$ the set $\{x \in X \mid x \mathrel{R} y\}$.

A set $A\subseteq\mathbb{R}$ without endpoints is $\aleph_1$-dense if it has cardinality $\aleph_1$ and between every two distinct elements of $A$ there are $\aleph_1$-many elements of $A$. Given an $\aleph_1$-dense $A\subseteq \mathbb{R}$, we denote its reverse by $A^* = \{-x \mid x \in A\}$. We say that $A$ is \emph{reversible} if $A \cong A^*$ as linear orders. 

Given some $p = (x,y) \in \mathbb{R}^2$, we let $p_0$ and $p_1$ be $x$ and $y$, respectively. Next, we introduce the following binary relations on $\mathbb{R}^2$: given $p,q \in \mathbb{R}^2$, let
\[
\begin{aligned}
p &\snearrow q \iff p_0 < q_0 \text{ and } p_1 < q_1, &
p &\ssearrow q \iff p_0 < q_0 \text{ and } p_1 > q_1,\\
p &\sswarrow q \iff p_0 > q_0 \text{ and } p_1 > q_1, &
p &\snwarrow q \iff p_0 > q_0 \text{ and } p_1 < q_1.
\end{aligned}
\]

\subsection{Abraham-Rubin-Shelah Open Coloring Axiom}\label{sec:ARS}
The principle $\mathsf{OCA}_{ARS}$, introduced by Abraham, Rubin, and Shelah in \cite{MR801036}, has the following statement:\vspace{0.5em}

\begin{center}
  \begin{minipage}{\dimexpr\linewidth-5em}
    If $X$ is a separable metrizable space of cardinality $\aleph_1$ and $c: [X]^2 \rightarrow \{1,\dots, n\}$ is a continuous map, then there are countably many $X_i$ ($i \in \omega$) such that $X = \bigcup_i X_i$ and $c$ is constant on each $[X_i]^2$.\vspace{0.5em}
  \end{minipage}
\end{center}

In particular, $\mathsf{OCA}_{ARS}$ implies that if $A\subseteq \mathbb{R}$ has size $\aleph_1$ and $f: A \rightarrow \mathbb{R}$ is injective, then $f$ is $\sigma$-monotone, i.e., it is the union of countably many monotone subfunctions. By an unpublished result of Todor\v{c}evi\'{c}  (reported in \cite{MR1978945}), the latter statement is actually equivalent to $\mathsf{OCA}_{ARS}$ under $\mathfrak{p} > \aleph_1$. 

The axioms $\mathsf{OCA}_T$ and $\mathsf{OCA}_{ARS}$ are mutually independent: $\mathsf{OCA}_{ARS}$ does not imply $\mathsf{OCA}_{T}$ (not even $\mathsf{OCA}_{T}(\aleph_1)$, the restriction of $\mathsf{OCA}_T$ to spaces of cardinality $\aleph_1$), since $\mathsf{OCA}_{T}(\aleph_1)$ refutes the existence of increasing sets \cite[Proposition 8.4]{MR980949} while $\mathsf{OCA}_{ARS}$ is consistent with their existence \cite[Theorem 3.1]{MR801036}; on the other hand, we also know that $\mathsf{OCA}_{T}$ alone does not imply $\mathsf{OCA}_{ARS}$ \cite[Theorem 3]{MR4228344}.

\subsection{Entangled and Increasing Sets}\label{sec:entangled}

Entangled sets were introduced in \cite{avraham1981martin} as particularly rigid linear orders witnessing the failure of $\mathsf{BA}$.  An uncountable set of reals $E \subseteq \mathbb{R}$ is said to be \emph{$n$-entangled}, for some positive integer $n > 0$, if for every uncountable collection $F$ of strictly increasing, pairwise disjoint $n$-tuples of elements of $E$ and for every $t \in {}^n 2$,
there are $x, y \in F$ such that, for every $i < n$, $x_i < y_i$ if and only if $t_i = 0$. An \emph{entangled} set is a set of reals which is $n$-entangled for every $n$.

Increasing sets were also introduced in \cite{avraham1981martin}, although they were explicitly defined in \cite{MR801036}. An uncountable set of reals $I \subseteq \mathbb{R}$ is said to be \emph{increasing} if for every positive integer $n$, for every uncountable collection $F$ of pairwise disjoint $n$-tuples of elements of $I$, there are $x, y \in F$ such that $x_i < y_i$ for every $i < n$. In particular, every entangled set is an increasing set. 

Todor{\v c}evi{\'c} \cite{MR799818} showed that entangled sets provide a powerful method to construct counterexamples relative to various problems,
such as the productivity of chain conditions and the square bracket partition relation.  The existence of entangled sets and increasing sets is independent of $\mathsf{ZFC}$: on one hand, these sets exist under $\mathsf{CH}$ \cite{MR799818, MR776283}; on the other hand, they do not exist under Baumgartner's Axiom. For recent applications and results regarding increasing and entangled sets, see, e.g., \cite{CARROY_LEVINE_NOTARO_2026, chapital2025nentangledsetn1entangledsets, martinezranero2025entangledsuslinlinesoga, marun2025baumgartnersaxiomsmallposets}.

Todor{\v c}evi{\'c}'s Open Coloring Axiom $\mathsf{OCA_T}$ implies that there are neither increasing sets nor $2$-entangled sets  \cite[Proposition 8.4]{MR980949}.
\section{Main result}\label{sec:main}

\begin{definition}\label{def:nsentangled}
A sequence $\vec{E} = \langle e_\xi \mid \xi < \omega_1\rangle$ of distinct reals is \emph{nonstationarily $2$-entangled} if $E = \mathrm{ran}(\vec{E})$ is $\aleph_1$-dense and, for every injective, monotone map without fixed points $f:E \rightharpoonup E$,  the set $\mathrm{dom}(f)$ is nonstationary with respect to $\vec{E}$.
\end{definition}

From now on, a subset $D \subseteq \mathrm{ran}(\vec{E})$ is (\emph{non})\emph{stationary with respect to $\vec{E}$} if $\{\xi < \omega_1 \mid e_\xi \in D\}$ is (non)stationary. Note that any enumeration of an $\aleph_1$-dense $2$-entangled set of reals is nonstationarily $2$-entangled.

It easily follows from Definition~\ref{def:nsentangled} that the range of a nonstationarily $2$-entangled sequence is not reversible. Hence, the existence of these sequences implies the failure of $\mathsf{BA}$. Moreover, their existence also implies the failure of $\mathsf{OCA}_{ARS}$:

\begin{lemma}\label{lemma:incomp}
If there exists a nonstationarily $2$-entangled sequence, then both $\mathsf{BA}$ and $\mathsf{OCA}_{ARS}$ fail. 
\end{lemma}
\begin{proof}
The failure of $\mathsf{BA}$ was discussed above. Towards showing that $\mathsf{OCA}_{ARS}$ also fails, fix a nonstationarily $2$-entangled sequence $\vec{E}$ and let $E = \mathrm{ran}(\vec{E})$. Pick an injective total map without fixed points $f: E \rightarrow E$. It quickly follows from $\vec{E}$ being nonstationarily $2$-entangled that $f$ cannot be $\sigma$-monotone. But $\mathsf{OCA}_{ARS}$ implies that $f$ is $\sigma$-monotone (see Section~\ref{sec:ARS}). Thus, $\mathsf{OCA}_{ARS}$ fails.
\end{proof}

Given Lemma~\ref{lemma:incomp} and the paragraph preceding its statement, both Theorem~\ref{thm:mainintro} and Corollary~\ref{cor:intro} follow from the next theorem. The rest of the section is devoted to its proof. 

\begin{theorem}\label{thm:main}
 It is consistent relative to $\mathsf{ZF}$ that $\mathsf{MA_{\aleph_1}} + \mathsf{OCA}_T$ holds and there exists a nonstationarily $2$-entangled sequence.
 \end{theorem}

Given a graph $G\subseteq [S]^2$ and a set $X \subseteq S$, we let $\mathcal{H}(X,G)$ be the poset of all finite subsets of $X$ which are $G$-cliques ordered by reverse inclusion. Moreover, given a nonstationarily $2$-entangled $\vec{E}$ and a forcing notion $\mathcal{P}$, we say that $\mathcal{P}$ \emph{preserves $\vec{E}$} if $\vec{E}$ is nonstationarily $2$-entangled in $V[H]$ for every $V$-generic filter $H \subseteq \mathcal{P}$.

 \begin{proposition}\label{prop:preservation}
Assume $\mathsf{CH}$.  Let $S$ be a separable metrizable space and let $G\subseteq [S]^2$ be an open graph with $\chi(G) > \aleph_0$. Then, for every nonstationarily $2$-entangled sequence $\vec{E}$, there exists an uncountable $Y \subseteq S$ such that $\mathcal{H}(Y,G)$ is ccc and preserves $\vec{E}$.
 \end{proposition}
 \begin{proof}
 The idea for the following forcing comes from \cite{MR4228344}. It is a refinement of Todor\v{c}evi\'{c}'s classical forcing from \cite[Theorem 4.4]{MR980949} (see also \cite[Theorem 2.1]{MR1233818}). 
 
Let $\langle M_\alpha \mid \alpha < \omega_1\rangle$ be a continuous $\in$-increasing chain of countable elementary submodels of $H(\aleph_2)$ such that $\mathbb{R} \subseteq \bigcup_{\alpha < \omega_1} M_\alpha$ and $M_0$ contains $S, G$, and $ \vec{E} = \langle e_\xi \mid \xi < \omega_1\rangle$. 

For each $r \in \mathbb{R}$, let $\mathrm{ht}(r)$ be the least $\alpha < \omega_1$ such that $r \in M_{\alpha+1}$. Note that $r \not\in M_{\mathrm{ht}(r)}$ whenever $r \not\in M_0$. 

Let $E = \mathrm{ran}(\vec{E})$ and let $C \subseteq \omega_1$ be the closed unbounded set of all those $\alpha \in \omega_1$ such that $M_\alpha \cap \omega_1 = \alpha$. Fix a sequence $\langle x_\alpha \mid \alpha \in C\rangle$ such that, for each $\alpha \in C$,
\begin{enumerate}
\item $x_\alpha \in S \cap (M_{\alpha+2} \setminus M_{\alpha+1})$, and
\item $x_\alpha \not\in A$ for every $G$-independent $A \in \mathcal{P}(S) \cap M_{\alpha+1}$.
\end{enumerate}
We can construct such a sequence because, for each $\alpha \in C$, \[M_{\alpha+2} \vDash ``\chi(G) > \aleph_0 \text{ and } M_{\alpha+1} \text{ is countable}."\] We now prove that $Y = \{x_\alpha \mid \alpha \in C\}$ satisfies the desired properties.

With minor modifications, the classical argument due to Todor\v{c}evi\'{c} \cite[Theorem 2.1]{MR1233818} shows that $\mathcal{H}(Y, G)$ is ccc (actually, powerfully ccc). We are left to prove that $\mathcal{H}(Y, G)$ preserves $\vec{E}$.

Let $K = \{e_\alpha \mid \alpha \in C\}$. We now prove that $\mathcal{H}(Y, G)$ forces the domain of every increasing, injective map without fixed points from a subset of $E$ to $E$ to have at most countable intersection with $K$. The decreasing case is analogous. This clearly suffices to show that $\mathcal{H}(Y, G)$ preserves $\vec{E}$.  In what follows, a \emph{witness} is an increasing, injective map without fixed points from a subset of $E$ to $E$.

To prove that $\mathcal{H}(Y, G)$ forces the domain of every witness to have at most countable intersection with $K$, fix any uncountable family $\mathcal{F}$ of pairs $(s, z)$ with $s \in \mathcal{H}(Y,G)$ and $z \in (K \times E)\setminus \{(e,e) \mid e \in K\}$, towards showing there are distinct $(s,z), (s',z') \in \mathcal{F}$ such that $s$ and $s'$ are compatible and $\{z,z'\}$ is non-increasing. 

We treat $\mathcal{F}$  as a binary relation over $\mathcal{H}(Y,G)$ and $K\times E$. By passing to an uncountable subset of $\mathcal{F}$ if necessary, we can suppose that there exists an $n \ge 0$ such that every condition in the domain of $\mathcal{F}$ has size $n$. We now proceed by induction on $n$. 

For each $s \in \mathrm{dom}(\mathcal{F})$, let $(s_0, \dots, s_{n-1})$ be the increasing enumeration of $s$ according to $\mathrm{ht}$. Identifying each element of $\mathrm{dom}(\mathcal{F})$ with an element of $S^{n}$, we may assume that some fixed basic open set $U$ in $S^{n}$ separates each element of $\mathrm{dom}(\mathcal{F})$. 

Let us first treat the base case $n = 0$:
\begin{claim}\label{claim:1}
For every witness $f$, $\mathrm{dom}(f) \cap K$ is at most countable.
\end{claim}
\begin{proof}
Let $F \subseteq E^2$ be the closure of the graph of $f$ in $E^2$. Fix a $\gamma < \omega_1$ such that $F \in M_\gamma$. Let $D$ be the set 
\begin{multline*}
\{x \in E \mid (x,x) \not\in F \text{ and } \exists! y ((x,y) \in F) \text{ and } \forall x' \forall y ( (x' < x \text{ and}\\ (x,y) \in F) \Rightarrow (x',y) \not\in F)\}.
\end{multline*}
It is easy to see that $F \cap (D\times E)$ is the graph of a witness with domain $D$. Clearly, $D \in M_\gamma$ and, by elementarity, there is a closed and unbounded subset $B \subseteq \omega_1$ which belongs to $M_\gamma$ and  which is disjoint from $\{\xi < \omega_1 \mid e_\xi \in D\}$. Now observe that $C \cap {[\gamma, \omega_1)} \subseteq B$: indeed, for every $\beta < \omega_1$ with $\beta \ge \gamma$, the set $B \cap M_\beta$ is unbounded in $\omega_1 \cap M_\beta$ by elementarity, and therefore $\omega_1 \cap M_\beta \in B$; since $M_\alpha \cap \omega_1 = \alpha$ for every $\alpha \in C$, we conclude that  $C \cap {[\gamma, \omega_1)} \subseteq B$. Thus, $D \cap K$ is at most countable. 

Now let us show that $\mathrm{dom}(f) \setminus D$ is at most countable. Pick some $x \in \mathrm{dom}(f)$. If $x\not\in D$, then one of the following two non-mutually exclusive cases must hold: either there is $y \in E$ such that $y \neq f(x)$ and $(x,y) \in F$; or there is some $x' \in E$ with $x' < x$ such that $(x', f(x)) \in F$. The first case holds only if $x$ is a discontinuity point of $f$. Since every monotone real map has at most countably many discontinuity points, we conclude that there are at most countably many $x \in \mathrm{dom}(f)$ that satisfy the first case. The second case holds only if there exists some $\epsilon > 0$  such that $(x-\epsilon, x) \cap \mathrm{dom}(f) = \emptyset$. But there are at most countably many such $x$'s in $\mathrm{dom}(f)$ by the separability of $\mathbb{R}$. Overall, $\mathrm{dom}(f) \setminus D$ is at most countable. 

Since both $D \cap K$ and $\mathrm{dom}(f) \setminus D$ are at most countable, we conclude that so is $\mathrm{dom}(f) \cap K$.
\end{proof}

It follows from Claim~\ref{claim:1} that if $n = 0$, then we can indeed find two distinct $z,z'$ such that $\{z,z'\}$ is non-increasing and $(\emptyset, z), (\emptyset, z') \in \mathcal{F}$, as otherwise $\mathrm{ran}(\mathcal{F})$ would be an uncountable witness with domain contained in $K$. So now let us proceed with the inductive step: suppose that our claim holds for $n$, towards showing that it also holds for $n+1$. 

Suppose that for some $s \in \mathrm{dom}(\mathcal{F})$, the set $\mathcal{F}_s$ is uncountable. Then, we are done, as by Claim~\ref{claim:1} there must be two distinct $z,z' \in K \times E$ such that $\{z,z'\}$ is non-increasing and $(s,z), (s,z') \in \mathcal{F}$. Analogously, if there is some $z$ such that $\mathcal{F}^z$ is uncountable, then, by the ccc property of $\mathcal{H}(Y,G)$, there would be two distinct compatible conditions $s,s'$ such that $(s,z), (s',z) \in \mathcal{F}$---the singleton $\{z\} \subseteq \mathbb{R}^2$ is vacuously non-increasing. So we restrict to the case in which $\mathcal{F}$ is the graph of a partial injective map from $[S]^{n+1} \cap \mathcal{H}(Y,G)$ into $K\times E$, and we treat $\mathcal{F}$ accordingly.

Now, if there exists some $e \in K$ such that $\mathcal{F}^{-1}(\{e\} \times E)$ is uncountable, we are done: indeed, it would contain two compatible conditions by the ccc property of $\mathcal{H}(Y,G)$. Analogously, we are done if $\mathcal{F}^{-1}(K \times \{e\})$  is uncountable for some $e \in E$. Hence we can further restrict to the case in which $\mathrm{ran}(\mathcal{F})$ is the graph of a partial injective map from $K$ into $E$.

There are two (non-mutually exclusive) cases that need to be dealt with. Recall that, given $z\in \mathbb{R}^2$, we write $z = (z_0, z_1)$.\vspace{1em}

\paragraph{\textbf{Case 1:}} \emph{there are uncountably many $(s,z) \in \mathcal{F}$ such that $\mathrm{ht}(s_n) > \mathrm{ht}(z_0)$.} By passing to the given uncountable subset of $\mathcal{F}$, assume that for all $(s,z) \in \mathcal{F}$, $\mathrm{ht}(s_n) > \mathrm{ht}(z_0)$. 

Given some $p = (s, z) \in S^n \times E^2$, some open $U \subseteq  S^n \times E^2$ such that $p \in U$, and some $d \in \{\ssearrow, \snwarrow\}$, let
\[
U_p^d = \big\{(s', z') \in U \mid z \mathrel{d} z' \text{ and } \{s_i, s'_i\} \in G \text{ for all }i < n\big\}.
\]
If $f$ is a partial function from $S^n \times E^2$ to $S$, $p \in S^n \times E^2$, and $d \in \{\ssearrow, \snwarrow\}$, let
\begin{align*}
\omega_f^{d}(p) &= \bigcap \big\{\mathrm{cl}(f[U_p^d]) \mid U\subseteq S^n \times E^2 \text{ open and } p \in U\big\},\\
\omega_f(p) &= \omega_f^{\ssearrow}(p) \cup \omega_f^{\snwarrow}(p).
\end{align*}
Consider the map $f:S^n \times E^2 \rightharpoonup S$ such that $f(s \upharpoonright n, z) = s_n$ for each $(s, z) \in \mathcal{F}$. 
\begin{claim}\label{claim:2}
There are at most countably many $p \in \mathrm{dom}(f)$ such that $f(p) \not\in \omega_f(p)$.
\end{claim}
\begin{proof}
Suppose otherwise, towards a contradiction. Then, there is a basic open $U \subseteq S^n \times E^2$ and a basic open $V \subseteq S$ such that for uncountably many $p \in \mathrm{dom}(f)$ we have $p \in U$ and $f(p) \in V$ and both $f[U_p^{\ssearrow}]$ and $f[U_p^{\snwarrow}]$ are disjoint from $V$. 

By induction hypothesis, there are distinct $p, p' \in \mathrm{dom}(f) \cap U$ with $p = (s, z)$ and $p' = (s', z')$ such that  $\{z, z'\}$ is non-increasing and $ s, s'$ are compatible.  Again, since we are assuming $\mathrm{ran}(\mathcal{F})$ to be the graph of an injective map, we have either $p' \in U_p^{\ssearrow}$ or $p' \in U_p^{\snwarrow}$. But then, either $f[U_p^{\ssearrow}] \cap V \neq \emptyset$ or $f[U_p^{\snwarrow}]\cap V \neq \emptyset$, which is a contradiction.
\end{proof}

Hence, by shrinking $\mathcal{F}$ if necessary, we can assume that $f(p) \in \omega_f(p)$ for all $p \in \mathrm{dom}(f)$. 

Let $f_0$ be a countable dense subfunction of $f$. We claim that $\omega_f = \omega_{f_0}$. Pick some $p \in S^n \times E^2$. Clearly, since $f_0$ is a subfunction of $f$, we have $\omega_{f_0}(p) \subseteq \omega_f(p)$. Now fix some $u \in \omega_f^{\ssearrow}(p)$, towards showing $u \in \omega_{f_0}^{\ssearrow}(p)$---the same argument holds for $\omega_f^{\snwarrow}(p)$. By definition, for every open neighborhood $U$ of $p$ and every open neighborhood $V$ of $u$, the set $\mathrm{graph}(f) \cap (U_p^{\ssearrow} \times V)$ is nonempty. Since $U_p^{\ssearrow}$ is an open subset of $S^n \times E^2$, and since $\mathrm{graph}(f_0)$ is a dense subset of $\mathrm{graph}(f)$, we conclude that $\mathrm{graph}(f_0) \cap (U_p^{\ssearrow} \times V)$ is also nonempty, and therefore $u \in \omega_{f_0}^{\ssearrow}(p)$.

Let $\gamma < \omega_1$ be such that $f_0 \in M_\gamma$, and pick some $p = (s,z) \in \mathrm{dom}(f)$ such that $\mathrm{ht} \circ f(p) \ge \gamma$. We have  $f(p) \in \omega_f(p) = \omega_{f_0} (p)$. Either $f(p) \in \omega_{f_0}^{\ssearrow}(p)$ or $f(p) \in \omega_{f_0}^{\snwarrow}(p)$ (or both). Suppose that $f(p) \in \omega_{f_0}^{\ssearrow}(p)$---the other case is analogous. In particular, we have 
\[
f(p) \in \bigcup_{a \in E} \omega_{f_0}^{\ssearrow}(s, (z_0, a)).
\]

Let $\alpha = \mathrm{ht} \circ f(p)$. The set $A =\bigcup_{a \in E} \omega_{f_0}^{\ssearrow}(s, (z_0, a))$ belongs to $M_{\alpha}$, as all the parameters involved in the definition of $A$ belong to $M_{\alpha}$. Since $f(p) \in A$, it must be, by property (2) of the sequence $\langle x_\beta \mid \beta \in C\rangle$, that $A$ is not a $G$-independent set. Pick $u,v \in A$ with $\{u,v\} \in G$ and  $b,c \in E$  with $b \le c$ such that $u \in \omega^{\ssearrow}_{f_0}(s, (z_0, b))$ and $v \in \omega^{\ssearrow}_{f_0}(s, (z_0, c))$. Pick also two disjoint open sets $U,V \subseteq S$ with $u \in U$ and $v \in V$ such that $U \times V \subseteq G$. Let $q = (t, x) \in \mathrm{dom}(f_0)$ be such that $(z_0, b) \ssearrow x$ and $s, t$ are disjoint and compatible and $f(q) \in U$. Pick some open $O \subseteq S^n$ such that $s \in O$ and every $t' \in O$ is compatible with $t$. Now pick $q' = (t', x') \in \mathrm{dom}(f_0)$ such that $t' \in O$  and $f(q') \in V$ and $ (z_0, c) \ssearrow x' \ssearrow x$. Then $\{x,x'\}$ is non-increasing, the two conditions $t \cup \{f(q)\}$ and $t' \cup \{f(q')\}$ are compatible and both $(t \cup \{f(q)\}, x)$ and $(t' \cup \{f(q')\}, x')$ belong to $\mathcal{F}$. This finishes Case 1.
\vspace{1em}

\paragraph{\textbf{Case 2:}} \emph{there are uncountably many $(s,z) \in \mathcal{F}$ such that $\mathrm{ht}(s_n) < \mathrm{ht}(z_0)$.} By passing to the given uncountable subset of $\mathcal{F}$, assume that for all $(s,z) \in \mathcal{F}$, $\mathrm{ht}(s_n) < \mathrm{ht}(z_0)$.

Given some $p \in S^{n+1}$, and some open $U \subseteq S^{n+1}$ such that $p \in U$, let
\[
U_p = \big\{q \in U \mid \forall i \le n \ \{p_i, q_i\} \in G\}.
\]
If $g$ is a partial function from $S^{n+1}$ to $E^2$ and $p \in S^{n+1}$, let
\[
\omega_g^d(p) = \big\{z \in E^2 \mid z \in \mathrm{cl}(g[U_p] \cap d_z)\text{ for all } U \subseteq S^{n+1} \text{ open and }p \in U\big\}
\]
for each $d\in \{\snearrow, \ssearrow, \sswarrow, \snwarrow\}$. Finally, let \[\omega_g(p) = \bigcup \big\{\omega_g^d(p) \mid d \in \{\snearrow, \ssearrow, \sswarrow, \snwarrow\}\big\}.\]

Consider the map $g:S^{n+1} \rightharpoonup K\times E$ such that $g(s) = z$ for all $(s,z) \in \mathcal{F}$---i.e., $\mathcal{F}$ is the graph of $g$.

It follows from the ccc property of $\mathcal{H}(Y,G)$ and from an argument analogous to the one in Claim~\ref{claim:2}, that there are at most countably many $p \in \mathrm{dom}(g)$ such that $g(p) \not\in \omega_g(p)$. Hence, shrinking $\mathcal{F}$ if necessary, we can assume that $g(p) \in \omega_g(p)$ for all $p \in \mathrm{dom}(g)$. 

Let $g_0$ be a countable dense subfunction of $g$, and let $\gamma < \omega_1$ be such that $g_0 \in M_\gamma$. An argument analogous to the one after Claim~\ref{claim:2} implies $\omega^d_g = \omega^d_{g_0}$ for every $d$. Pick some $p \in \mathrm{dom}(g)$ such that $\mathrm{ht}( g(p)_0) \ge \gamma$. Let $\alpha = \mathrm{ht}(g(p)_0)$. Note that $\alpha \in C$ and $g(p)_0 = e_\alpha$. There exists a $d\in \{\snearrow, \ssearrow, \sswarrow, \snwarrow\}$ such that $g(p) \in \omega_{g_0}^d(p)$. Suppose that $g(p) \in \omega_{g_0}^{\snearrow}(p)$---the other three cases are analogous.
\begin{claim}\label{claim:3}
There are distinct $z,z' \in \omega^{\snearrow}_{g_0}(p)$ such that $\{z,z'\}$ is non-increasing.
\end{claim}
\begin{proof}
Suppose otherwise, towards a contradiction. Then, $\omega_{g_0}^{\snearrow}(p)$ would be the graph of an injective and increasing partial map. Consider the set $T = \{z \in \omega_{g_0}^{\snearrow}(p) \mid z_0 \neq z_1\}$. Note that $T$ is the graph of a witness, since we have avoided the possible fixed points of $\omega_{g_0}^{\snearrow}(p)$. Moreover,  since $g(p) \in \omega_{g_0}^{\snearrow}(p)$ and $g(p)_0 \neq g(p)_1$, we also have $g(p) \in T$.

Since all the parameters involved in the definition of $ \omega^{\snearrow}_{g_0}(p)$ (and hence of $T$) are in $M_{\alpha}$, we conclude that $T \in M_{\alpha}$. Hence, there exists a club $B\subseteq \omega_1$ with $B \in M_{\alpha}$ which is disjoint from $\{\xi < \omega_1 \mid e_\xi \in \mathrm{dom}(T)\}$. By elementarity, $\omega_1 \cap M_{\alpha} = \alpha \in B$. But since $g(p) \in T$, in particular $g(p)_0 = e_{\alpha}$ is in the domain of $T$. Hence the contradiction.
\end{proof}

Fix $z,z'$ given by Claim~\ref{claim:3}. Now, either $z \ssearrow z'$, or $z \snwarrow z'$, or $z_0 = z'_0$, or $z_1 = z'_1$. We treat only the case $z_1 = z'_1$ and $z_0 < z'_0$ as the others are treated analogously. Pick some $q \in \mathrm{dom}(g_0)$ disjoint from $p$ such that $q$ and $p$ are compatible and $z \snearrow g(q)$ and $g(q) \ssearrow z'$. Fix a $U\subseteq S^{n+1}$ such that $p \in U$ and every $s \in U$ is compatible with $q$. Now pick $r \in \mathrm{dom}(g_0) \cap U$ such that $ z' \snearrow g(r)$ and $g(q) \ssearrow g(r)$. Both pairs $(q, g(q))$ and $(r, g(r))$ belong to $\mathcal{F}$, the conditions $q, r$ are compatible, and $\{g(q),g(r)\}$ is non-increasing. This finishes Case 2 and, with it, the proof.
 \end{proof}

Proposition~\ref{prop:preservation} tells us that, under $\mathsf{CH}$, we can always force an uncountable clique in an uncountably chromatic open graph while preserving a nonstationarily $2$-entangled sequence. This would already suffice to prove that $\mathsf{OCA}_{T}$ is consistent with the existence of a nonstationarily $2$-entangled sequence. The next proposition is key to forcing $\mathsf{MA_{\aleph_1}}$ on top of $\mathsf{OCA}_{T}$ via Abraham and Shelah's explicit contradiction method.

 \begin{proposition}\label{prop:explicit}
 Assume $\mathsf{CH}$. Let $\vec{E}$ be a nonstationarily $2$-entangled sequence and let $\mathcal{P}$ be a ccc poset which does not preserve $\vec{E}$. Then, there is a ccc poset $\mathcal{Q}$ which preserves $\vec{E}$ and forces $\mathcal{P}$ not to be ccc.
 \end{proposition}
 \begin{proof}
 Fix a nonstationarily $2$-entangled sequence $\vec{E} = \langle e_\xi \mid \xi < \omega_1\rangle$ and let $\mathcal{P}$ be a ccc poset that does not preserve $\vec{E}$. Let $E = \mathrm{ran}(\vec{E})$. There must exist a $\mathcal{P}$-name $\dot{f}$ and a condition $p \in \mathcal{P}$ such that
\begin{multline*}
 p \Vdash \dot{f}: E \rightharpoonup  E  \text{ is an injective monotone map without fixed points}\\\text{and } \mathrm{dom}(\dot{f}) \text{ is stationary with respect to } \vec{E}.
\end{multline*}
Let us also suppose that $p$ forces $\dot{f}$ to be increasing (the other case is analogous). The set
\[
X = \big\{\xi < \omega_1 \mid \exists q \in \mathcal{P}\ \exists a \in E \ (q \le p \text{ and } q \Vdash \dot{f}(e_\xi) = a)\big\}
\]
is stationary. For each $\xi \in X$, pick $p_\xi \le p$ and $a_\xi \in E$ such that $p_\xi \Vdash \dot{f}(e_\xi) = a_\xi$. Since $p$ forces $\dot{f}$ not to have fixed points, we have $a_\xi \neq e_\xi$ for each $\xi \in X$.

Consider the set $W = \{(e_\xi, a_\xi) \mid \xi \in X\} \subseteq E^2$ and let $H\subseteq [W]^2$ be the following open graph: $\{z,z'\} \in H$ if and only if $z,z' \in W$ and either $z \ssearrow z'$ or $z \snwarrow z'$. We claim that $H$ is uncountably chromatic. 
\begin{claim}
$\chi(H) > \aleph_0$.
\end{claim}
\begin{proof}
Suppose otherwise, towards a contradiction, and let $(f_n)_{n\in\omega}$ be a sequence of increasing maps such that $\bigcup_n f_n = W$. We now prove that for each $n\in\omega$, there exists an at most countable set $D_n \subseteq E$ such that $f_n \upharpoonright (E\setminus D_n)$ is injective.

First note that for each $a \in E$ the set $L_a = \{\xi \in X \mid a_\xi = a\}$ is at most countable, since $\{p_\xi \mid \xi \in L_a\}$ is an antichain of $\mathcal{P}$.

Moreover, as $f_n$ is monotone, there are at most countably many $a \in E$ such that the set $f_n^{-1}(\{a\})$ has more than one element. Hence, the set
\[
D_n = \bigcup \big\{f_n^{-1}(\{a\}) \mid a \in E \text{ and } |f_n^{-1}(\{a\})| > 1\big\}
\]
is at most countable, and $f_n \upharpoonright (E\setminus D_n)$ is injective.

Since we are assuming $\vec{E}$ to be nonstationarily $2$-entangled, this means that $\mathrm{dom}(f_n)$ is nonstationary with respect to $\vec{E}$ for each $n$, which is a contradiction, as $X$ would be the union of countably many nonstationary sets. 
\end{proof}
 
 By Proposition~\ref{prop:preservation}, there exists a ccc poset $\mathcal{Q}$ that preserves $\vec{E}$ and adds an uncountable $H$-clique over $W$---that is, it adds an uncountable, injective, and decreasing map whose graph is contained in $W$. But this means that $\mathcal{Q}$ forces the existence of an uncountable antichain in $\mathcal{P}$. 
 \end{proof}
 
Now we are ready to prove Theorem~\ref{thm:main}.
 \begin{proof}[Proof of Theorem~\ref{thm:main}]
We start with a model of $V=L$ and we let $\vec{E}$ be any enumeration of some $\aleph_1$-dense $2$-entangled set---recall that the existence of an entangled set follows from $\mathsf{CH}$ (see Section~\ref{sec:entangled}). 

Let us briefly recall the well-known finite support iteration $\langle \mathcal{P}_\alpha, \dot{\mathcal{Q}}_\alpha \mid \alpha < \omega_2\rangle$ that forces $\mathsf{MA_{\aleph_1}}+\mathsf{OCA}_T$ (see, e.g., \cite[p. 142]{MR1233818}): first fix a sequence witnessing $\diamondsuit (\{\alpha < \omega_2 \mid \mathrm{cof}(\alpha) = \omega_1\})$; if $\alpha$ has countable cofinality, we let $\dot{\mathcal{Q}}_\alpha$ be the $\mathcal{P}_\alpha$-name designated by the bookkeeping function---whose role is to ensure that $\mathsf{MA}_{\aleph_1}$ is forced---which is forced to be a ccc poset; if $\alpha$ has uncountable cofinality, and $\dot{G}$ is the $\mathcal{P}_\alpha$-name, guessed by our diamond sequence, which is forced to be an uncountably chromatic open graph over some separable metrizable space, we let  $\dot{\mathcal{Q}}_\alpha$ be the $\mathcal{P}_\alpha$-name of some ccc poset that adds an uncountable clique to $\dot{G}$. However, we have an additional concern: to preserve $\vec{E}$. Next, we see how we modify the iteration in order to address this problem.

Suppose that $\alpha$ has countable cofinality and that $\mathcal{P}_\alpha$ is defined and preserves $\vec{E}$. Let $\dot{\mathcal{R}}$ be the designated $\mathcal{P}_\alpha$-name for a ccc poset given by our bookkeeping function. Using a routine mixing argument, we can devise $\dot{\mathcal{Q}}_\alpha$ such that, for every $p \in \mathcal{P}_\alpha$, if $p \Vdash_{\mathcal{P}_\alpha} ``\dot{\mathcal{R}}$ preserves $\vec{E}$", then $p \Vdash_{\mathcal{P}_\alpha} \dot{\mathcal{Q}}_\alpha = \dot{\mathcal{R}}$; and if $p \Vdash_{P_\alpha} ``\dot{\mathcal{R}}$ does not preserve $\vec{E}$", then $p$ forces $\dot{\mathcal{Q}}_\alpha$ to be ccc, preserve $\vec{E}$, and destroy the ccc property of $\dot{\mathcal{R}}$---in the latter case we use Proposition~\ref{prop:explicit}.

Now suppose that $\alpha$ has uncountable cofinality and that $\mathcal{P}_\alpha$ is defined and preserves $\vec{E}$. Let $\dot{\mathcal{Q}}_\alpha$ be the $\mathcal{P}_\alpha$-name for the poset given by Proposition~\ref{prop:preservation} that preserves $\vec{E}$ and adds an uncountable clique to the uncountably chromatic open graph guessed by our diamond sequence. 

We show that $\mathcal{P}_\alpha$ preserves $\vec{E}$ for all $\alpha \le \omega_2$, by induction on $\alpha$. If $\alpha$ is successor, then the preservation of $\vec{E}$ follows by inductive hypothesis and by construction.

Now suppose that $\alpha$ is limit of countable cofinality. Let $G\subseteq\mathcal{P}_\alpha$ be a $V$-generic filter and fix some injective monotone map $f:E\rightharpoonup E$ without fixed points in $V[G]$. There are $(f_n)_{n\in\omega}$ such that $f = \bigcup_n f_n$ and $f_n \in V[G \cap \mathcal{P}_{\alpha_n}]$ for some $\alpha_n < \alpha$. By induction hypothesis, the domain of each $f_n$ is nonstationary with respect to $\vec{E}$ (in $V[G \cap \mathcal{P}_{\alpha_n}]$ and, a fortiori, in $V[G]$). Hence, the domain of $f$ is nonstationary with respect to $\vec{E}$ in $V[G]$ and $\mathcal{P}_\alpha$ preserves $\vec{E}$.

Now suppose that $\alpha$ is limit of uncountable cofinality. Let $G\subseteq\mathcal{P}_\alpha$ be a $V$-generic filter and fix some injective monotone map $f:E\rightharpoonup E$ without fixed points in $V[G]$. If we let $F$ be the closure of the graph of $f$, there exists some $\beta < \alpha$ such that $F \in V[G \cap \mathcal{P}_\beta]$. Arguing as in Claim~\ref{claim:1}, there is a set $D\subseteq E$ in $V[G \cap \mathcal{P}_\beta]$ such that $F \cap (D\times E)$ is the graph of an injective, monotone map without fixed points and $\mathrm{dom}(f) \setminus D$ is at most countable. By induction hypothesis, $D$ is nonstationary with respect to $\vec{E}$ in $V[G \cap \mathcal{P}_\beta]$, and, a fortiori, in $V[G]$. We conclude that $\mathrm{dom}(f)$ is also nonstationary with respect to $\vec{E}$ in $V[G]$. Overall, $\mathcal{P}_\alpha$ preserves $\vec{E}$ and we are done.
 \end{proof}
\printbibliography

@article{CARROY_LEVINE_NOTARO_2026, 
title={Some questions on entangled linear orders}, 
JOURNAL = {J. Symb. Log.},
  FJOURNAL = {The Journal of Symbolic Logic},
author={Carroy, Rapha\"{e}l and Levine, Maxwell and Notaro, Lorenzo}, 
year={2026}, 
pages={1–27},
note         = {To appear},
      eprint={2507.17503},
  eprinttype   = {arxiv},
}

@misc{marun2025baumgartnersaxiomsmallposets,
      title={Baumgartner's Axiom and Small Posets}, 
      author={Pedro Marun and Saharon Shelah and Corey Bacal Switzer},
      year={2025},
      eprint={2512.21247},
      archivePrefix={arXiv}
}

@book {MR1940513,
    AUTHOR = {Jech, Thomas},
     TITLE = {Set Theory},
    SERIES = {Springer Monographs in Mathematics},
   EDITION = {The Third Millennium Edition},
 PUBLISHER = {Springer-Verlag, Berlin},
      YEAR = {2003},
      ISBN = {3-540-44085-2}
}

@misc{martinezranero2025entangledsuslinlinesoga,
      title={Entangled Suslin lines and $\mathsf{OGA}$}, 
      author={Carlos Martinez-Ranero and Lucas Polymeris},
      year={2025},
      eprint={2512.01065},
      archivePrefix={arXiv}
}

@misc{chapital2025nentangledsetn1entangledsets,
      title={There may be an $n$-entangled set but no $n+1$-entangled sets}, 
      author={Jorge Antonio Cruz Chapital},
      year={2025},
      eprint={2509.01029},
      archivePrefix={arXiv}
}

@article {MR1978945,
    AUTHOR = {Moore, Justin T.},
     TITLE = {Weak diamond and open colorings},
   JOURNAL = {J. Math. Log.},
  FJOURNAL = {Journal of Mathematical Logic},
    VOLUME = {3},
      YEAR = {2003},
    NUMBER = {1},
     PAGES = {119--125},
      ISSN = {0219-0613,1793-6691}
}

@incollection {MR776640,
    AUTHOR = {Baumgartner, James E.},
     TITLE = {Applications of the proper forcing axiom},
 BOOKTITLE = {Handbook of set-theoretic topology},
     PAGES = {913--959},
 PUBLISHER = {North-Holland, Amsterdam},
      YEAR = {1984},
      ISBN = {0-444-86580-2}
}

@incollection {MR1233818,
    AUTHOR = {Veli\u{c}kovi\'{c}, Boban},
     TITLE = {Applications of the open coloring axiom},
 BOOKTITLE = {Set theory of the continuum ({B}erkeley, {CA}, 1989)},
    SERIES = {Math. Sci. Res. Inst. Publ.},
    VOLUME = {26},
     PAGES = {137--154},
 PUBLISHER = {Springer, New York},
      YEAR = {1992},
      ISBN = {0-387-97874-7}
}

@article {MR4228344,
    AUTHOR = {Moore, Justin T.},
     TITLE = {Some remarks on the {O}pen {C}oloring {A}xiom},
   JOURNAL = {Ann. Pure Appl. Logic},
  FJOURNAL = {Annals of Pure and Applied Logic},
    VOLUME = {172},
      YEAR = {2021},
    NUMBER = {5},
     PAGES = {Paper No. 102912, 6},
      ISSN = {0168-0072, 1873-2461}
}

@article {MR1440716,
    AUTHOR = {Farah, Ilijas},
     TITLE = {OCA and towers in {$\mathcal {P}(\mathbb{N})/{\mathrm{fin}}$}},
   JOURNAL = {Comment. Math. Univ. Carolin.},
  FJOURNAL = {Commentationes Mathematicae Universitatis Carolinae},
    VOLUME = {37},
      YEAR = {1996},
    NUMBER = {4},
     PAGES = {861--866},
      ISSN = {0010-2628, 1213-7243}
}

@article {MR4642565,
    AUTHOR = {Guzm\'{a}n, Osvaldo and Hru\v{s}\'{a}k, Michael and Koszmider,
              Piotr},
     TITLE = {Almost disjoint families and the geometry of nonseparable
              spheres},
   JOURNAL = {J. Funct. Anal.},
  FJOURNAL = {Journal of Functional Analysis},
    VOLUME = {285},
      YEAR = {2023},
    NUMBER = {11},
     PAGES = {Paper No. 110149, 49},
      ISSN = {0022-1236, 1096-0783}
}

@article {MR4288216,
    AUTHOR = {Braga, Bruno M. and Farah, Ilijas and Vignati, Alessandro},
     TITLE = {Uniform {R}oe coronas},
   JOURNAL = {Adv. Math.},
  FJOURNAL = {Advances in Mathematics},
    VOLUME = {389},
      YEAR = {2021},
     PAGES = {Paper No. 107886, 35},
      ISSN = {0001-8708, 1090-2082}
}

@article {MR4290495,
    AUTHOR = {McKenney, Paul and Vignati, Alessandro},
     TITLE = {Forcing axioms and coronas of {$\mathrm{C}^*$}-algebras},
   JOURNAL = {J. Math. Log.},
  FJOURNAL = {Journal of Mathematical Logic},
    VOLUME = {21},
      YEAR = {2021},
    NUMBER = {2},
     PAGES = {Paper No. 2150006, 73},
      ISSN = {0219-0613, 1793-6691}
}

@article {MR1999941,
    AUTHOR = {Todor\v{c}evi\'{c}, Stevo},
     TITLE = {A proof of Nogura's conjecture},
   JOURNAL = {Proc. Amer. Math. Soc.},
  FJOURNAL = {Proceedings of the American Mathematical Society},
    VOLUME = {131},
      YEAR = {2003},
    NUMBER = {12},
     PAGES = {3919--3923},
      ISSN = {0002-9939, 1088-6826}
}

@article {MR2776359,
    AUTHOR = {Farah, Ilijas},
     TITLE = {All automorphisms of the Calkin algebra are inner},
   JOURNAL = {Ann. of Math. (2)},
  FJOURNAL = {Annals of Mathematics. Second Series},
    VOLUME = {173},
      YEAR = {2011},
    NUMBER = {2},
     PAGES = {619--661},
      ISSN = {0003-486X, 1939-8980}
}

@article {MR1202874,
    AUTHOR = {Veli\u{c}kovi\'{c}, Boban},
     TITLE = {$\mathsf{OCA}$ and automorphisms of ${\mathcal{P}}(\omega)/{\mathrm{fin}}$},
   JOURNAL = {Topology Appl.},
  FJOURNAL = {Topology and its Applications},
    VOLUME = {49},
      YEAR = {1993},
    NUMBER = {1},
     PAGES = {1--13},
      ISSN = {0166-8641, 1879-3207}
}

@article {MR3385104,
    AUTHOR = {Chodounsk\'y, David and Zapletal, Jind\v{r}ich},
     TITLE = {Why {Y}-c.c.},
   JOURNAL = {Ann. Pure Appl. Logic},
  FJOURNAL = {Annals of Pure and Applied Logic},
    VOLUME = {166},
      YEAR = {2015},
    NUMBER = {11},
     PAGES = {1123--1149},
      ISSN = {0168-0072, 1873-2461}
}

@article {MR4128471,
    AUTHOR = {Miyamoto, Tadatoshi and Yorioka, Teruyuki},
     TITLE = {A fragment of {A}sper\'o-{M}ota's finitely proper forcing
              axiom and entangled sets of reals},
   JOURNAL = {Fund. Math.},
  FJOURNAL = {Fundamenta Mathematicae},
    VOLUME = {251},
      YEAR = {2020},
    NUMBER = {1},
     PAGES = {35--68},
      ISSN = {0016-2736, 1730-6329}
}

@article {MR4819970,
    AUTHOR = {Guzm\'an, Osvaldo and Todor\v{c}evi\'{c}, Stevo},
     TITLE = {The {$P$}-ideal dichotomy, {M}artin's axiom and entangled
              sets},
   JOURNAL = {Israel J. Math.},
  FJOURNAL = {Israel Journal of Mathematics},
    VOLUME = {263},
      YEAR = {2024},
    NUMBER = {2},
     PAGES = {909--963},
      ISSN = {0021-2172, 1565-8511}
}

@article {MR317934,
    AUTHOR = {Baumgartner, James E.},
     TITLE = {All {$\aleph_1$}-dense sets of reals can be isomorphic},
   JOURNAL = {Fund. Math.},
  FJOURNAL = {Polska Akademia Nauk. Fundamenta Mathematicae},
    VOLUME = {79},
      YEAR = {1973},
    NUMBER = {2},
     PAGES = {101--106},
      ISSN = {0016-2736, 1730-6329}
}

@article {MR776283,
    AUTHOR = {Bonnet, Robert and Shelah, Saharon},
     TITLE = {Narrow {B}oolean algebras},
   JOURNAL = {Ann. Pure Appl. Logic},
  FJOURNAL = {Annals of Pure and Applied Logic},
    VOLUME = {28},
      YEAR = {1985},
    NUMBER = {1},
     PAGES = {1--12},
      ISSN = {0168-0072, 1873-2461}
}

@article {MR801036,
    AUTHOR = {Abraham, Uri and Rubin, Matatyahu and Shelah, Saharon},
     TITLE = {On the consistency of some partition theorems for continuous
              colorings, and the structure of {$\aleph_1$}-dense real order
              types},
   JOURNAL = {Ann. Pure Appl. Logic},
  FJOURNAL = {Annals of Pure and Applied Logic},
    VOLUME = {29},
      YEAR = {1985},
    NUMBER = {2},
     PAGES = {123--206},
      ISSN = {0168-0072}
}

@book {MR980949,
    AUTHOR = {Todor\v{c}evi\'{c}, Stevo},
     TITLE = {Partition problems in topology},
    SERIES = {Contemporary Mathematics},
    VOLUME = {84},
 PUBLISHER = {American Mathematical Society, Providence, RI},
      YEAR = {1989},
     PAGES = {xii+116},
      ISBN = {0-8218-5091-1}
}

@article{avraham1981martin,
	author = {Abraham, Uri and Shelah, Saharon},
	journal = {Israel Journal of Mathematics},
	number = {1},
	pages = {161--176},
	publisher = {Springer},
	title = {Martin's axiom does not imply that every two $\aleph_1$-dense sets of reals are isomorphic},
	volume = {38},
	year = {1981}}

@article {MR799818,
    AUTHOR = {Todor\v{c}evi\'c, Stevo},
     TITLE = {Remarks on chain conditions in products},
   JOURNAL = {Compositio Math.},
  FJOURNAL = {Compositio Mathematica},
    VOLUME = {55},
      YEAR = {1985},
    NUMBER = {3},
     PAGES = {295--302},
      ISSN = {0010-437X, 1570-5846}
}

\end{document}